\documentclass[12pt]{amsart}
\usepackage{amsmath,amssymb,amsthm,upref,graphicx,mathrsfs}

\usepackage{color}
\usepackage[
  colorlinks=true,
  linkcolor=blue,
  citecolor=blue,
  urlcolor=blue]{hyperref}

\numberwithin{equation}{section}

\newtheorem{theorem}{Theorem}[section]

\newtheorem{lemma}[theorem]{Lemma}
\newtheorem{cor}[theorem]{Corollary}

\theoremstyle{definition}
\newtheorem{definition}[theorem]{Definition}

\newtheorem{remark}[theorem]{Remark}

\newcommand{\be}{\begin{eqnarray*}}
\newcommand{\ee}{\end{eqnarray*}}
\newcommand{\beq}{\begin{equation}}
\newcommand{\eeq}{\end{equation}}

\begin{document}

\title[A generalized dyadic maximal operator involving the infinite product]
  {Weighted inequalities
for a generalized dyadic maximal operator involving the infinite product}

\authors

\author[W. Chen]{Wei Chen}
\address{Wei Chen \\ School of Mathematical Sciences,
Yangzhou University, 225002 Yangzhou, China}
\email{weichen@yzu.edu.cn}

\author[R. J. Chen]{Ruijuan Chen}
\address{Ruijuan Chen \\ School of Mathematical Sciences,
Yangzhou University, 225002 Yangzhou, China}
\email{rjchenyzu@gmail.com}

\author[C. Zhang]{ Chao Zhang}
\address{Chao Zhang \\ School of Statistics and Mathematics, Zhejiang Gongshang University, 310018 Hangzhou , China}
\email{zaoyangzhangchao@163.com}

\makeatletter
\renewcommand{\@makefntext}[1]{#1}
\makeatother \footnotetext{\noindent
 Supported by the National Natural
Science Foundation of China (Grant No. 11101353 and No. 11271292),
the Natural Science Foundation of Jiangsu Education Committee (Grant No. 11KJB110018), the Natural Science Foundation of Jiangsu Province (Grant No. BK2012682) and the Natural Science Foundation of Zhejiang Province (Grant No. LQ14A010001).}


%
%
\begin{abstract} {We define a generalized dyadic maximal operator involving the infinite product
and discuss weighted inequalities for the operator. A formulation of the Carleson embedding theorem is proved.
Our results depend heavily on a generalized H\"{o}lder's inequalities.}
\end{abstract}

\maketitle

%
%
 \section{Introduction }
\subsection{Weighted Inequalities for the Hardy-Littlewood Maximal Operator and the Multisubliear One in $R_n$}
Let $R_n$ be the $n\hbox{-dimensional}$ real Euclidean
space and $f$ a real valued measurable function. The classical
Hardy-Littlewood maximal operator $M$ is defined by
\be Mf(x)=\sup\limits_{x\in Q}\frac{1}{|Q|}\int_Q|f(y)|dy,\ee
where $Q$ is a non-degenerate cube with its sides parallel to the coordinate
axes and $|Q|$ is the Lebesgue measure of $Q.$

Let $u,~v$ be two weights, i.e., positive locally integrable functions. As is
well known, for $p\geq1,$ Muckenhoupt \cite{Muckenhoupt B} showed that the
inequality
$$ \lambda^p\int_{\{Mf>\lambda\}}u(x)dx
\leq C \int_{R_n}|f(x)|^pv(x)dx, ~~\lambda>0,~f\in{L^p(v)}
$$
holds if and only if $(u,v)\in A_p,$ i.e., for any cube $Q$ in
$R_n$ with sides parallel to the coordinates
$$
\big(\frac{1}{|Q|}\int_Qu(x)dx\big)
\big(\frac{1}{|Q|}\int_Qv^{-\frac{1}{p-1}}(x)dx\big)^{p-1}<C,~p>1;
$$
$$
\frac{1}{|Q|}\int_Qu(x)dx\leq C \mathop{\hbox{ess inf}}\limits_{Q}v(x)
,~p=1.
$$
Suppose that $u=v$ and $p>1,$ Muckenhoupt \cite{Muckenhoupt B} also proved that
$$ \int_{R_n}\big(Mf(x)\big)^pv(x)dx
\leq C \int_{R_n}|f(x)|^pv(x)dx, ~\forall f\in{L^p(v)}
$$
holds if and only if $v$ satisfies
\begin{equation}\label{Ap}
\big(\frac{1}{|Q|}\int_Qv(x)dx\big)
\big(\frac{1}{|Q|}\int_Qv^{-\frac{1}{p-1}}(x)dx\big)^{p-1}<C,~\forall Q.
\end{equation}
But, the problem of finding all $u$ and $v$ such that
$$ \int_{R_n}\big(Mf(x)\big)^pu(x)dx
\leq C \int_{R_n}|f(x)|^pv(x)dx,~\forall f\in{L^p(v)}
$$
is much hard and complicated. In order to solve the problem,
Sawyer \cite{Sawyer E T.} established the
testing condition $S_{p,q},$ i.e., for any cube $Q$ in
$R_n$ with sides parallel to the coordinates
$$\Big(\int_{Q}\big(M(\chi_Qv^{1-p'})(x)\big)^qu(x)dx\Big)^{\frac{1}{q}}
\leq C(\int_Qv^{1-p'}(x)dx)^{\frac{1}{p}},$$
where $1<p\leq q<\infty.$
The condition $S_{p,q}$ is a
sufficient and necessary condition such that
the weighted inequality
$$\Big(\int_{R_n}\big(Mf(x)\big)^qu(x)dx\Big)^{\frac{1}{q}} \leq C
\Big(\int_{R_n}|f(x)|^pv(x)dx\Big)^{\frac{1}{p}}, ~\forall f\in{L^p(v)}$$
holds.
Motivated by these
results, the theory of weighted inequalities developed rapidly in the last years, not
only for the Hardy-Littlewood maximal operator but also for some of the main
operators in Harmonic Analysis like Calderson-Zygmund operators (see \cite{Cruz-Uribe D. J. M. Martell} and \cite{Garcia Rubio}
for more informations).

Recently, the multisublinear maximal function
\begin{equation}\label{multi_maximal_operator}\mathcal{M}(f_1,...,f_m)(x) = \sup\limits_{x\in Q}
\prod\limits_{i=1}\limits^{m}\frac{1}{|Q|}\int_Q|f_i(y_i)|dy_i
\end{equation}
associated with cubes with sides parallel to the coordinate
axes was studied in \cite{Lerner A.K. Ombrosi S.}.
The importance of this operator is that it generalizes the Hardy--Littlewood
maximal function (case $m=1$) and in several ways it controls the class
of multilinear Calderon--Zygmund operators as it is shown in \cite{Lerner A.K. Ombrosi S.}.
The relevant class of multiple weights for $\mathcal{M}$ is given by the condition $A_{\overrightarrow{p}}:$ for
$\overrightarrow{p}=(p_1,p_2,\cdot\cdot\cdot,~ p_m),$
$\overrightarrow{\omega}=(\omega_1, ~\omega_2,\cdot\cdot\cdot,~\omega_m)$ and a weight $v,$
the weight vector $(v, \overrightarrow{\omega})\in A_{\overrightarrow{p}}$ if
$$\sup_Q\frac{v(Q)}{|Q|}\prod\limits^m_{i=1}\big(\frac{1}{|Q|}\int_Q\omega_i^{-\frac{1}{p_i-1}}(y_i)dy_i\big)^{\frac{p}{p'_i}}
< \infty,$$
where $\frac{1}{p}=\sum\limits^m_{i=1}\frac{1}{p_i }$ and $1\leq p_1,p_2,...,p_m<\infty.$

It is easy to see that in the linear case (that is,
if $m=1$), condition $A_{\overrightarrow{p}}$ is the usual $A_p.$
In \cite{Lerner A.K. Ombrosi S.} the following multilinear extension
of the Muckenhoupt $A_p$ theorem for the maximal function was obtained:  the inequality
$$\|\mathcal{M}(\overrightarrow{f})\|_{L^{p,\infty}(v)}\leq
C\prod\limits^m_{i=1}\|f_i\|_{L^{p_i}(\omega_i)},
~\forall f_i\in L^{p_i}(\omega_i)$$
holds if and only if $(v, \overrightarrow{\omega})\in A_{\overrightarrow{p}}.$
Moreover, if $1< p_1,p_2,...,p_m<\infty$ and
$v=\prod_{i=1}^mw_i^{p/p_i},$ then the inequality
$$\|\mathcal{M}(\overrightarrow{f})\|_{L^{p}(v)}\leq
C\prod\limits^m_{i=1}\|f_i\|_{L^{p_i}(\omega_i)},
~\forall f_i\in L^{p_i}(\omega_i)$$
holds if and only if $(v, \overrightarrow{\omega})\in A_{\overrightarrow{p}}.$ The more general case was extensively discussed in
\cite{Grafakos Perez, Grafakos Liu}.

In order to establish the generalization of Sawyer's theorem to the multilinear setting,
a kind of monotone property and a
reverse H\"{o}lder's inequality on the weights were introduced in \cite{W. M. Li} and \cite{Chen-Damian}, respectively.
They both obtained the multilinear version of Sawyer's result.

In this paper, for suitable $\overrightarrow{f}=(f_1,f_2,...)$(see Remarks \ref{lem def} and \ref{proper 2} for two kinds of suitable conditions), we define a new generalized dyadic maximal function
\begin{equation}\label{g-multi_maximal_operator}\mathfrak{M}_d(\overrightarrow{f})(x)\triangleq\sup\limits_{x\in B\in\mathcal{D}}
\prod\limits_{i=1}\limits^{\infty}\frac{1}{|B|}\int_B|f_i(y_i)|dy_i,
\end{equation}
where $\mathcal{D}$ is the family of dyadic cubes in $R_n.$
This operator involves the concept of an infinite product which will be recalled in Section \ref{section B0}. Our main result are weighted inequalities for the operator. In addition, we can define the following operator
$$\mathfrak{M}(\overrightarrow{f})(x)\triangleq\sup\limits_{x\in Q}
\prod\limits_{i=1}\limits^{\infty}\frac{1}{|Q|}\int_Q|f_i(y_i)|dy_i
$$
associated with cubes with sides parallel to the coordinate
axes. Then it is natural to establish weighted inequalities for it.
But, the method of \cite[Lemma 2.2]{Lerner A.K. Ombrosi S.} is not suitable.
One reason is that Calderon--Zygmund decomposition
deeply depends on the constant $m,$ which appears in \eqref{multi_maximal_operator}.
However, this is not the end of the story.
The related theory in martingale setting was established in \cite{Chen-Jiao}.

Our paper is organized as follows. Section \ref{pre} contains
some basic definitions and facts concerning series, Lebesgue's integral,
the infinity product and dyadic cubes needed throughout the rest of this paper.
This section also contains a generalized H\"{o}lder's inequalities borrowed from \cite{Chen-Jiao}.
Our main results are stated and proved in Section \ref{m re}.

\section{Preliminaries}\label{pre}
\subsection{Some Properties of Series, Lebesgue's Integral and the Infinity Product}\label{section B0}
Let $\{a_i\}$ be a sequence of real numbers.
Let $\{s_n\}$ be the sequence obtained from $\{a_i\},$ where for each $n\in N,~s_n=\sum\limits^{n}_{i=1}a_i.$
If $s_n$ converges in $R$ or diverges to $+\infty$ (or $-\infty$),
we say that the sum of the series is well defined and we denote the sum as $\sum\limits^{\infty}_{i=1}a_i.$
Let $\lambda_i\in(0,1),b_i\in R, ~i\in N,$ and let $\sum\limits_{i=1}^{\infty}\lambda_i=1.$ It is known that $(N, 2^N)$ is a measurable space.
By the sequences $\{\lambda_i\}$ and $\{b_i\},$
we can define a measure $\lambda$ and a measurable function $b$ on the space in the following way
$$\lambda(i)=\lambda_i\hbox{ and }b(i)=b_i,~\forall i\in N.$$ Then $(N, 2^N, \lambda)$ is a probability space. Applying Levi's Lemma, we have
$$\sum\limits^{\infty}_{i=1}\lambda_ib^+_i=\lim\limits_{k\rightarrow\infty}\sum\limits^k_{i=1}\lambda_ib^+_i
=\lim\limits_{k\rightarrow\infty}\int_N b^+\chi_{\{1,2,...,k\}} d\lambda=\int_N b^+d\lambda$$ and $$\sum\limits^{\infty}_{i=1}\lambda_ib^-_i=\lim\limits_{k\rightarrow\infty}\sum\limits^k_{i=1}\lambda_ib^-_i
=\lim\limits_{k\rightarrow\infty}\int_N b^-\chi_{\{1,2,...,k\}} d\lambda=\int_N b^-d\lambda.$$
For simplicity, we denote $\sum\limits^{\infty}_{i=1}\lambda_ib^+_i$ and $\sum\limits^{\infty}_{i=1}\lambda_ib^-_i$ by $A$ and $B,$ respectively.  It follows that $A,B\in [0,+\infty].$ If $A$ or $B$ is finite,
then $\sum\limits^{\infty}_{i=1}\lambda_ib_i$ is well defined, integral of $b$ exists and
$$\sum\limits^{\infty}_{i=1}\lambda_ib_i=\int_N bd\lambda.$$

Let us recall the concept of an infinite product(see, e.g., \cite[p.~298]{Rudin}).

\begin{definition} Suppose $\{c_n\}$ is a sequence of complex number,
$$p_n=\prod\limits^{n}_{i=1}c_i,$$
and $p=\lim\limits_{n\rightarrow\infty}p_n$ exists. Then we write
\begin{equation}\label{Rudin}p=\prod\limits^{\infty}_{i=1}c_i.\end{equation}
The $p_n$ are the partial products of the infinite product (\ref{Rudin}).
We should say that the infinite product (\ref{Rudin}) converges if the sequence $\{p_n\}$ converges.
\end{definition}

\begin{remark} \label{re-aa}Suppose $\{c_n\}$ and $\{c'_n\}$ are nonnegative sequences,
and the infinite product $\prod\limits^{\infty}_{i=1}c_i$ converges. If $c'_n\leq c_n,~n\in N,$ then the infinite product $\prod\limits^{\infty}_{i=1}c'_i$ also converges.
\end{remark}

\begin{remark} \label{re-bb}Suppose $\{f_i\}$ is a sequence of measurable functions on a measurable space $(\Omega,\mathcal{F}),$ and suppose that
the sequence of numbers $\{\prod\limits^{n}_{i=1}f_i(x)\}$ converges for every $x\in \Omega.$ We can then define a function $\prod\limits^{\infty}_{i=1}f_i$ by
$$\prod\limits^{\infty}_{i=1}f_i(x)=\lim\limits_{n\rightarrow\infty}\prod\limits^{n}_{n=1}f_i(x).$$
We should say that the function $\prod\limits^{\infty}_{i=1}f_i(x)$ is well defined.
\end{remark}

\begin{remark}\label{C-P-series-cor-E}
Let $1<p_i<\infty, i\in N$ and $\sum\limits_{i=1}^{\infty}\frac{1}{p_i}=\frac{1}{p}.$
Then
$$\prod\limits^{\infty}_{i=1}p'_i<\infty,$$
where $\frac{1}{p_i}+\frac{1}{p'_i}=1,~i\in N.$
\end{remark}

The above Remark \ref{C-P-series-cor-E} can be checked easily(see, e.g., \cite[Theorem 2.12]{Chen-Jiao}). In addition, we mention that the assumptions in  following Remark \ref{remark con} can be found in \cite[P.68]{L. Grafakos}.

\begin{remark}\label{remark con}Let $1<p_i<\infty, i\in N$ and $\frac{1}{p}=\sum\limits^{\infty}_{i=1}\frac{1}{p_i }.$ If $\sum\limits^{\infty}_{i=1}\frac{\ln p_i}{p_i}<\infty,$ then $\prod\limits^{\infty}_{i=1}p^{\frac{p'}{p}} p'<\infty.$
\end{remark}

\subsection{Generalized H\"{o}lder's Inequality for Integral}\label{section B 1}
In the subsection,
we suppose that $(\Omega,\mathcal{F},\mu)$ is a measure space and $\{f_i\}$ is a sequence of nonnegative measurable functions on $(\Omega,\mathcal{F},\mu)$. We recall the following lemma which is a
generalized H\"{o}lder's inequalities(see, e.g., \cite[Theorem 2.11]{Chen-Jiao}). This kind of inequality also be discussed on the $\sigma-$finite measure space in \cite{Karakostas}.

\begin{lemma}\label{C-P-prop-A}
Let $0<p_i<\infty, i\in N$ and $\sum\limits_{i=1}^{\infty}\frac{1}{p_i}=\frac{1}{p}.$
If $\prod\limits^{\infty}_{i=1}\|f_i\|_{L^{p_i}}<\infty,$ then the function $\prod\limits^{\infty}_{i=1}f_i$ is well defined and $\|\prod\limits^{\infty}_{i=1}f_i\|_{L^p}\leq \prod\limits^{\infty}_{i=1}\|f_i\|_{L^{p_i}}.$
\end{lemma}

\subsection{Dyadic Cubes and the Dyadic Maximal Function in $R_n$}\label{dya cube}
 In $R_n$, let $[0,1)^n$ be the unit cube open on the right, and let $\mathcal{D}_0$ be the collection of such cubes
with vertices lying on the lattice $Z^n.$ Dilating the family by a factor of $2^k,~ k\in Z,$ we get a family $\mathcal{D}_k$ of
dyadic cubes whose vertices lie on the lattice $(2^{-k}Z)^n.$ We call
the members of $\mathcal{D}=\bigcup\limits_{k\in Z}\mathcal{D}_k$ dyadic cubes. Given a cube $B\in \mathcal{D},$ we denote by $|B|$ its Lebesgue measure. Observe that two dyadic cubes are either disjoint, or one is contained in the other. For each $x\in R_n$ and $k\in Z$, there is a unique element of $\mathcal{D}_k$ containing $x.$ Moreover, the $\sigma-$algebra $\sigma(\mathcal{D}_k)$ of measurable subsets of $R_n$ formed by countable
unions and complements of elements of $\mathcal{D}_k$ is increasing as $k$ increases.

Given a locally integral function on $R_n,$ we define its dyadic maximal function $M_d(f)$ by
$$M_d(f)(x)=\sup\limits_{x\in B\in\mathcal{D}}\frac{1}{|B|}\int_B|f(y)|dy.$$
Recall that the conditional expectation of a locally integrable function $f$ on $R_n$ with respect
to the increasing family of $\sigma(\mathcal{D}_k)$ is defined as(see, e.g., \cite[P.~384]{L. Grafakos})
$$E_k(f)(x) = \sum\limits_{B\in \mathcal{D}_k}(\frac{1}{|B|}\int_Bf(y)dy)\chi_{B}(x).$$
Then, we have $M_d(f)(x)=\sup\limits_{k\in Z}E_k(|f|)(x).$ Moreover, for suitable $\overrightarrow{f}=(f_1,f_2,...),$ we also have $\mathfrak{M}_d(\overrightarrow{f})(x)=\sup\limits_{k\in Z}
\prod\limits_{i=1}\limits^{\infty}E_k(|f_i|)(x).$

\begin{remark}\label{lem def} Let $p_i>1,~i\in N.$
If $\prod\limits^{\infty}_{i=1}\|f_i\|_{L^{p_i}}<\infty$ and $\sum\limits_{i=1}^{\infty}\frac{1}{p_i}=\frac{1}{p},$
then $$ \|\mathfrak{M}_d(\overrightarrow f)\|_{L^p}\leq\|\prod\limits^{\infty}_{i=1}M_df_i\|_{L^p}\leq \prod\limits_{i=1}^{\infty}\|M_df_i\|_{L^{p_i}}\leq \big(\prod\limits^{\infty}_{i=1}p'_i\big)\prod\limits_{i=1}^{\infty}\|f_i\|_{L^{p_i}}<\infty,$$
where we have used Lemma \ref{C-P-prop-A} and Remark \ref{C-P-series-cor-E}.
\end{remark}

\section{Main Results and Proofs}\label{m re}

There are a lot of assumptions and notations which will be used in the section.
For convenience, we state them at the beginning of this part.

\textbf{Assumptions and Notations}
Let $\omega_i\in L^1_{loc}$ and $1<p_i<\infty$, $i\in N,$ and let $\{f_i\}$ be a sequence of nonnegative measurable function on $R_n.$
Write $\overrightarrow{p}=(p_1,p_2,\cdot\cdot\cdot),$
$\overrightarrow{\omega}=(\omega_1, \omega_2,\cdot\cdot\cdot),$
$\overrightarrow{f}=(f_1,f_2,...)$ and $\sigma_i=\omega_i^{-\frac{1}{p_i-1}},~i\in N.$ In addition, we also write $\overrightarrow{f\chi_G}=(f_1\chi_G,f_2\chi_G,...)$
and $\overrightarrow{\sigma\chi_G}
=(\sigma_1\chi_G,\sigma_2\chi_G,\cdot\cdot\cdot),$ where $G$ is a measurable set.

We suppose that
$\frac{1}{p}=\sum\limits^{\infty}_{i=1}\frac{1}{p_i },$ $\prod\limits_{i=1}^{\infty}E_k(\omega_i^{1-p'_i})^{\frac{1}{p'_i}}<\infty$
and $\prod\limits_{i=1}^{\infty}\sigma_i^{\frac{1}{p_i}}>0.$ We also suppose that $\prod\limits_{i=1}^{\infty}\|f_i\|_{L^{p_i}(\omega_i)}<\infty$ and denote it by $\overrightarrow{f}\in \prod\limits_{i=1}^{\infty}{L^{p_i}(\omega_i)}.$ Moreover, for all $B\in \mathcal{D},$ we assuming that $\overrightarrow{\sigma\chi_B}\in \prod\limits_{i=1}^{\infty}{L^{p_i}(\omega_i)}$ and denote it by $\overrightarrow{\sigma}\in \prod\limits_{i=1}^{\infty}{L_{loc}^{p_i}(\omega_i)}.$

\begin{remark}\label{proper 2} It follows from generalized H\"{o}lder's inequality for integral that $$\int_{R_n}\prod\limits_{i=1}^{\infty}E_k(f_i^{p_i}\omega_i)^{\frac{p}{p_i}}dx
\leq\prod\limits_{i=1}^{\infty}\big(\int_{R_n} E_k(f_i^{p_i}\omega_i)dx\big)^{\frac{p}{p_i}}
=\prod\limits_{i=1}^{\infty}\big(\int_{R_n} f_i^{p_i}\omega_idx\big)^{\frac{p}{p_i}}<\infty.$$
Hence,
$\prod\limits_{i=1}^{\infty}E_k(f_i^{p_i}\omega_i)^{\frac{1}{p_i}}<\infty.$
By H\"{o}lder's inequality and Remark \ref{re-aa}, we have
\be\prod\limits_{i=1}^{\infty}E_k(f_i)
   &\leq&\prod\limits_{i=1}^{\infty}E_k(f_i^{p_i}\omega_i)^{\frac{1}{p_i}}
         E_k(\omega_i^{-\frac{1}{p_i-1}})^{\frac{1}{p'_i}}\\
   &=&\prod\limits_{i=1}^{\infty}E_k(f_i^{p_i}\omega_i)^{\frac{1}{p_i}}
         \prod\limits_{i=1}^{\infty}E_k(\omega_i^{-\frac{1}{p_i-1}})^{\frac{1}{p'_i}}<\infty.\ee
Then $\mathfrak{M}_d(\overrightarrow{f})$ is well defined. Moreover, for $B\in \mathcal{D},$
we have $\prod\limits_{i=1}^{\infty}E_k(\sigma_i\chi_B)<\infty$
and $\mathfrak{M}_d(\overrightarrow{\sigma\chi_B})$ is well defined.\end{remark}

\subsection{Generalized $A_p$ Weights Involving the Infinite Product}

\begin{theorem}\label{theorem_Ap} Let $1<p_i<\infty, i\in N$ and $\frac{1}{p}=\sum\limits^{\infty}_{i=1}\frac{1}{p_i }.$
  Let $v$ and $\omega_i$ be weights. Then the following statements
are equivalent:
  \begin{enumerate}
    \item \label{theorem_Ap_3}There exists a positive constant $C$ such that
       $$(\frac{1}{|B|}\int_Bv(x)dx)^{\frac{1}{p}}
        \prod\limits_{i=1}^{\infty}(\frac{1}{|B|}\int_B\omega_i^{-\frac{1}{p_i-1}}(y_i)dy_i)^{\frac{1}{p_i'}}\leq C,~\forall B\in\mathcal{D};$$
    \item \label{theorem_Ap_1} There exists a positive constant $C$ such that
        $$v(B)^{\frac{1}{p}}\prod\limits_{i=1}^{\infty}(\frac{1}{|B|}\int_Bf_i(y_i)dy_i)\leq
        C\prod\limits^{\infty}_{i=1}\|f_i\chi_B\|_{L^{p_i}(\omega_i)},~\forall \overrightarrow{f}\in \prod\limits_{i=1}^{\infty}{L^{p_i}(\omega_i)},~B\in\mathcal{D};$$
    \item \label{theorem_Ap_2} There exists a positive constant $C$ such that
       $$\lambda v(\{\mathfrak{M}_d(\overrightarrow{f})\geq\lambda\})^{\frac{1}{p}}\leq C\prod\limits^{\infty}_{i=1}\|f_i\|_{L^{p_i}(\omega_i)},~\forall \overrightarrow{f}\in \prod\limits_{i=1}^{\infty}{L^{p_i}(\omega_i)},\ \lambda>0;$$
    \item \label{theorem_Ap_21} There exists a positive constant $C$ such that
       $$\lambda v(\{\mathfrak{M}_d(\overrightarrow{f})>\lambda\})^{\frac{1}{p}}\leq C\prod\limits^{\infty}_{i=1}\|f_i\|_{L^{p_i}(\omega_i)},~\forall \overrightarrow{f}\in \prod\limits_{i=1}^{\infty}{L^{p_i}(\omega_i)},\ \lambda>0.$$
  \end{enumerate}
  Moreover, we denote the smallest constants $C$ in \eqref{theorem_Ap_3}, \eqref{theorem_Ap_1}, \eqref{theorem_Ap_2} and \eqref{theorem_Ap_21}
by $[v,\overrightarrow{\omega}]_{A_{\overrightarrow{p}}},$ $[v,\overrightarrow{\omega}]'_{A_{\overrightarrow{p}}},$ $\|\mathfrak{M}_d\|'$ and $\|\mathfrak{M}_d\|,$ respectively.
Then it follows that
$$[v,\overrightarrow{\omega}]_{A_{\overrightarrow{p}}}=[v,\overrightarrow{\omega}]'_{A_{\overrightarrow{p}}}=\|\mathfrak{M}_d\|'=\|\mathfrak{M}_d\|.$$
  \end{theorem}

\begin{proof} We shall follow the scheme: $\eqref{theorem_Ap_3}\Leftrightarrow\eqref{theorem_Ap_1},$ $\eqref{theorem_Ap_2}\Leftrightarrow\eqref{theorem_Ap_21}$ and $\eqref{theorem_Ap_2}\Rightarrow\eqref{theorem_Ap_1}\Rightarrow\eqref{theorem_Ap_21}.$
And $\eqref{theorem_Ap_2}\Leftrightarrow\eqref{theorem_Ap_21}$ is trivial.

\eqref{theorem_Ap_3}$\Rightarrow$\eqref{theorem_Ap_1} For $B\in \mathcal{D},$
it follows from H\"{o}lder's inequality and \eqref{theorem_Ap_3} that
\be           &&v(B)^{\frac{1}{p}}\prod\limits_{i=1}^{\infty}(\frac{1}{|B|}\int_B f_i(y_i)dy_i)\\
              &\leq&v(B)^{\frac{1}{p}}\prod\limits_{i=1}^{\infty}(\frac{1}{|B|}\int_B f_i^{p_i}(y_i)\omega_i(y_i) dy_i)^{\frac{1}{p_i}}
                      (\frac{1}{|B|}\int_B\omega_i^{-\frac{p'_i}{p_i}}(y_i)dy_i)^{\frac{1}{p'_i}}\\
              &=&\prod\limits_{i=1}^{\infty}(\int_B f_i^{p_i}(y_i)\omega_i(y_i) dy_i)^{\frac{1}{p_i}}
                      \Big((\frac{1}{|B|}\int_Bv(x)dx)^{\frac{1}{p}}\prod\limits_{i=1}^{\infty}(\frac{1}{|B|}
                      \int_B\omega_i^{-\frac{1}{p_i-1}}(y_i)dy_i)^{\frac{1}{p'_i}}\Big)\\
              &\leq& [v,\overrightarrow{\omega}]_{A_{\overrightarrow{p}}}\prod\limits^{\infty}_{i=1}\|f_i\chi_B\|_{L^{p_i}(\omega_i)}. \ee

$\eqref{theorem_Ap_1}\Rightarrow\eqref{theorem_Ap_3}$ Let $f_i=\omega_i^{-\frac{1}{p_i-1}}.$ For $B\in \mathcal{D},$
we have \be(\frac{1}{|B|}\int_Bv(x)dx)^{\frac{1}{p}}
                  \prod\limits_{i=1}^{\infty}(\frac{1}{|B|}\int_B\omega_i^{-\frac{1}{p_i-1}}(y_i)dy_i)
        &=&(\frac{1}{|B|})^{\frac{1}{p}}v(B)^{\frac{1}{p}}\prod\limits_{i=1}^{\infty}(\frac{1}{|B|}\int_B f_i(y_i)dy_i)\\
        &\leq&[v,\overrightarrow{\omega}]'_{A_{\overrightarrow{p}}}(\frac{1}{|B|})^{\frac{1}{p}}\prod\limits_{i=1}^{\infty}
                 (\int_B\omega_i^{-\frac{1}{p_i-1}}(y_i)dy_i)^{\frac{1}{p_i}}\\
        &=&[v,\overrightarrow{\omega}]'_{A_{\overrightarrow{p}}}\prod\limits_{i=1}^{\infty}
                 (\frac{1}{|B|}\int_B\omega_i^{-\frac{1}{p_i-1}}(y_i)dy_i)^{\frac{1}{p_i}}.\ee
It follows that
\be(\frac{1}{|B|}\int_Bv(x)dx)^{\frac{1}{p}}
                  \prod\limits_{i=1}^{\infty}(\frac{1}{|B|}\int_B\omega_i^{-\frac{1}{p_i-1}}(y_i)dy_i)^{\frac{1}{p'_i}}
        \leq [v,\overrightarrow{\omega}]'_{A_{\overrightarrow{p}}}.\ee

$\eqref{theorem_Ap_2}\Rightarrow\eqref{theorem_Ap_1}$ Let $B\in \mathcal{D}.$ For $x\in B,$ we have
$$\prod\limits_{i=1}^{\infty}(\frac{1}{|B|}\int_Bf_i(y_i)dy_i)\leq\mathfrak{M}_d(\overrightarrow{f\chi_B})(x).$$
It follows from $\eqref{theorem_Ap_2}$ that
\be \prod\limits_{i=1}^{\infty}(\frac{1}{|B|}\int_B f_i(y_i)dy_i)v(B)^{\frac{1}{p}}
       &\leq&\lambda v(\{\mathfrak{M}_d(\overrightarrow{f\chi_B})\geq\lambda\})^{\frac{1}{p}}\\
       &\leq& \|\mathfrak{M}_d\|'\prod\limits^{\infty}_{i=1}\|f_i\chi_B\|_{L^{p_i}(\omega_i)},\ee
where $\lambda=\prod\limits_{i=1}^{\infty}(\frac{1}{|B|}\int_B f_i(y_i)dy_i).$

\eqref{theorem_Ap_1}$\Rightarrow$\eqref{theorem_Ap_21} For $R>0,$ we shall denote by $\mathfrak{M}_d^{(R)}(\overrightarrow{f}),$
the maximal operators obtained by taking in the corresponding definition just those cubes whose side length is less than or equal to $R.$
Since $\mathfrak{M}_d(\overrightarrow{f})(x)=\lim\limits_{R\rightarrow\infty}\mathfrak{M}_d^{(R)}(\overrightarrow{f})(x)$ and $\mathfrak{M}_d^{(R)}(\overrightarrow{f})(x)$ increases with $R.$ It will be enough to prove the inequality for $\mathfrak{M}_d^{(R)}$ with constant independent of $R.$ But, after fixing $R>0$ and $\lambda>0,$ observe that $\{x\in R_n:\mathfrak{M}_d^{(R)}(\overrightarrow{f})(x)>\lambda\}=\bigcup\limits_{j}Q_j,$ where $Q_j$ are the maximal dyadic cubes of side length less than or equal to $R$ for which $$\prod\limits_{i=1}^{\infty}(\frac{1}{|Q_j|}\int_{Q_j}f_i(y_i)dy_i)>\lambda.$$
These maximal dyadic cubes do exist because of the restriction on their size. Moreover, they are disjoint. It follows from \eqref{theorem_Ap_1} and the generalized H\"{o}lder's inequality that
\be \lambda^p v(\{x\in R_n:\mathfrak{M}_d^{(R)}(\overrightarrow{f})>\lambda\})
             &=&\lambda^p\sum\limits_{j}v(Q_j)\\
             &\leq&\sum\limits_{j}v(Q_j)(\prod\limits_{i=1}^{\infty}(\frac{1}{|Q_j|}\int_{Q_j}f_i(y_i)dy_i))^p\\
             &\leq&([v,\overrightarrow{\omega}]'_{A_{\overrightarrow{p}}})^p\sum\limits_{j}\prod\limits^{\infty}_{i=1}\|f_i\chi_{Q_j}\|_{L^{p_i}(\omega_i)}^p\\
             &\leq&([v,\overrightarrow{\omega}]'_{A_{\overrightarrow{p}}})^p\prod\limits^{\infty}_{i=1}(\sum\limits_{j}\int_{Q_j}f_i^{p_i}(y_i)\omega_i(y_i)dy_i)^{\frac{p}{p_i}}\\
             &\leq&([v,\overrightarrow{\omega}]'_{A_{\overrightarrow{p}}}\prod\limits^{\infty}_{i=1}\|f_i\|_{L^{p_i}(\omega_i)})^p.\ee
Thus
$$\lambda v(\{\mathfrak{M}(\overrightarrow{f})>\lambda\})^{\frac{1}{p}}\leq [v,\overrightarrow{\omega}]'_{A_{\overrightarrow{p}}}\prod\limits^{\infty}_{i=1}\|f_i\|_{L^{p_i}(\omega_i)}.$$
\end{proof}

The following Theorem \ref{ttheorem_Sp} is essentially taken from J. Garc\'{i}a-Cuerva and J. L. Rubio de Francia \cite[P.~423]{Garcia Rubio}.
In this paper, we refine the result by a limit process. Moreover, combining with Remark \ref{remark con}, we can get Corollary \ref{coro_Sp}.

\begin{theorem}\label{ttheorem_Sp}
Let $\omega$ be a weight and $1<
p<\infty.$ Suppose that $\sigma=\omega^{-\frac{1}{p-1}}\in L^1_{loc},$ then the
following statements are equivalent:
\begin{enumerate}
\item \label{tthm A 1}There exists a positive constant $C$ such that
\be
\|M_d(f)\|_{L^p(\omega)}\leq
C\|f\|_{L^{p}(\omega)},
~\forall f\in L^{p}(\omega);
\ee
\item \label{tthm A 3} There exists a positive constant $C$ such that
$$\frac{\omega(B)}{|B|}(\frac{\sigma(B)}{|B|})^{p-1}\leq C,~\forall B\in \mathcal{D}.$$
\end{enumerate}
Moreover, we denote the smallest constants $C$ in \eqref{tthm A 1}
and \eqref{tthm A 3} by $\|M_d\|$ and $[\omega]_{A_p},$ respectively.
Then it follows that
$$[\omega]_{A_p}
=\|M_d\|.$$
\end{theorem}

\begin{proof} It is clear that $\eqref{tthm A 1}\Rightarrow\eqref{tthm A 3},$
so we omit it.

$\eqref{tthm A 3}\Rightarrow\eqref{tthm A 1}$ It suffices to prove $\eqref{tthm A 1}$ for nonnegative functions.
Let $f$ be a nonnegative function in
$L^{p}(\omega)$ and let $\alpha>1.$ For every integer $k,$ we shall consider the set
$$S_k=\{x\in R_n: \alpha^k<M_df\leq \alpha^{k+1}\}.$$
From the definition of $M_df,$ $S_k\subseteq\bigcup\limits_{j}B_{k,j},$ where $B_{k,j}\in \mathcal{D}$ satisfies
$$\frac{1}{|B_{k,j}|}\int_{B_{k,j}}f(y)dy>\alpha^k.$$
Define $E_{k,1}=B_{k,1}\cap S_k$ and for $j>1:$
$$E_{k,j}=(B_{k,j}\backslash\bigcup\limits_{s<j}B_{k,s})\cap S_k.$$
The sets $S_k$ form a disjoint collection and each $S_k$ is disjoint union of the sets $E_{k,j}$ for varying $j.$
It follows that
\be
\int_{R_n}M_d(f)^p \omega dx
   &\leq&\alpha^p\sum\limits_{k\in Z,j\in Z}\omega(E_{k,j})(\frac{1}{|B_{k,j}|}\int_{B_{k,j}}f(x)dx)^p\\
   &=&\alpha^p\sum\limits_{k\in Z,j\in Z}\mu_{k,j}g_{k,j},\ee
where $$\mu_{k,j}=\omega(E_{k,j})(\frac{\sigma(B_{k,j})}{|B_{k,j}|})^p\hbox{ and } g_{k,j}=(\frac{1}{\sigma(B_{k,j})}\int_{B_{k,j}}f(y)\sigma(y)^{-1}\sigma(y)dy)^p.$$
We view the sum $\sum\limits_{k\in Z,j\in Z}\mu_{k,j}g_{k,j}$ as an integral on a measure space $(X,\mu)$ built over the set
$X=(\{(k,j)\})$ by assigning to each $(k,j)$ the measure $\mu_{k,j}.$ For $\lambda>0,$ call
$$\Gamma(\lambda)=\{(k,j)\in X:g_{k,j}>\lambda\}$$
$$G(\lambda)=\bigcup\limits_{(k,j)\in \Gamma(\lambda)}B_{k,j}.$$
Then $\sum\limits_{k\in Z,j\in Z}\mu_{k,j}g_{k,j}=\int_0^{\infty}\mu(\Gamma(\lambda))d\lambda.$
Observe that $\eqref{tthm A 3}$ is equivalent to saying that
$$(\frac{\sigma(B_{k,j})}{|B_{k,j}|})^p\leq [\omega]_{A_p}^{p'}(\frac{|B_{k,j}|}{\omega(B_{k,j})})^{p'}.$$
We use this to estimate $\mu_{k,j},$
\be \mu_{k,j}&\leq&[\omega]_{A_p}^{p'}\omega(E_{k,j})(\frac{|B_{k,j}|}{\omega(B_{k,j})})^{p'}\\
             &\leq&[\omega]_{A_p}^{p'}\int_{E_{k,j}}M_d^\omega(\chi_{B_{k,j}}\omega^{-1})^{p'}(x)\omega(x)dx.
\ee
Next we shall use the boundedness of $M_d^\omega$ to estimate $\mu(\Gamma(\lambda)).$

Then
\be\mu(\Gamma(\lambda))&=&\sum\limits_{(k,j)\in \Gamma(\lambda)}\mu_{k,j}\\
                       &\leq&[\omega]_{A_p}^{p'}\sum\limits_{(k,j)\in \Gamma(\lambda)}\int_{E_{k,j}}M_d^\omega(\chi_{B_{k,j}}\omega^{-1})^{p'}(x)\omega(x)dx\\
                       &\leq&[\omega]_{A_p}^{p'}\int_{G(\lambda)}M_d^\omega(\chi_{G(\lambda)}\omega^{-1})^{p'}(x)\omega(x)dx\\
                       &\leq&[\omega]_{A_p}^{p'}p^{p'}\int_{G(\lambda)}\sigma(x)dx\\
                       &\leq&[\omega]_{A_p}^{p'}p^{p'}
                       \int_{\{(M_d^\sigma(f\sigma^{-1}))^p>\lambda\}}\sigma(x)dx.
                       \ee
Now,
\be
\int_{R_n}M_d(f)^p \omega dx
   &\leq&\alpha^p\int_0^{\infty}\mu(\Gamma(\lambda))d\lambda\\
   &\leq&\alpha^pp^{p'}[\omega]_{A_p}^{p'}
                       \int_0^{\infty}\int_{\{(M_d^\sigma(f\sigma^{-1}))^p>\lambda\}}\sigma(x)dx d\lambda\\
   &=&\alpha^pp^{p'}[\omega]_{A_p}^{p'}
                       \int_{R_n} M_d^\sigma(f\sigma^{-1})^p\sigma dx\\
   &\leq&\alpha^pp^{p'}{p'}^p[\omega]_{A_p}^{p'}
                       \int_{R_n}f^p\omega dx.
\ee
Then we can take the limit $\alpha\rightarrow 1,$ which gives
$$(\int_\Omega M_d(f)^p \omega dx)^{\frac{1}{p}}\leq [\omega]_{A_p}^{\frac{ p'}{p}} p^{\frac{p'}{p}} p' (\int_\Omega f^p\omega dx)^{\frac{1}{p}}.$$
\end{proof}

\begin{definition} \label{product Ap}Let $1<p_i<\infty, i\in N$ and $\frac{1}{p}=\sum\limits^{\infty}_{i=1}\frac{1}{p_i }.$
Let $\omega_i\in A_{p_i},~i\in N.$
We say that the weight vector $\overrightarrow{\omega}$
satisfies the condition $A^*_{\overrightarrow{p}}$ involving the infinite product, if
$$\prod\limits^{\infty}_{i=1}[\omega_i]_{A_{p_i}}^{\frac{ 1}{p_i}}<\infty.$$ \end{definition}

\begin{cor}\label{coro_Sp} Let $1<p_i<\infty, i\in N$ and $\frac{1}{p}=\sum\limits^{\infty}_{i=1}\frac{1}{p_i }.$
If $\sum\limits^{\infty}_{i=1}\frac{\ln p_i}{p_i}<\infty$ and the weight vector $\overrightarrow{\omega}$
satisfies the condition $A^*_{\overrightarrow{p}},$
then
\be
\|\mathfrak{M}_d(\overrightarrow{f})\|_{L^p(v)}\leq
C\prod\limits^{\infty}_{i=1}\|f_i\|_{L^{p_i}(\omega_i)},
~\forall \overrightarrow{f}\in \prod\limits_{i=1}^{\infty}{L^{p_i}(\omega_i)},
\ee
where $v=\prod\limits_{i=1}^{\infty}\omega_i^{\frac{ 1}{p_i}}$ and $C=\prod\limits^{\infty}_{i=1}[\omega_i]_{A_{p_i}}^{\frac{ p_i'}{p_i}}
\prod\limits^{\infty}_{i=1}{p_i}^{\frac{p_i'}{p_i}}\prod\limits^{\infty}_{i=1}p_i'.$
\end{cor}

\subsection{Generalized $S_p$ Weight Involving the Infinite Product}
\begin{definition} \label{Rh}Let $1<p_i<\infty, i\in N$ and $\frac{1}{p}=\sum\limits^{\infty}_{i=1}\frac{1}{p_i }.$
Let $\omega_i$ be weights and let $\sigma_i=\omega_i^{-\frac{1}{p_i-1}},~i\in N.$
We say that the weight vector $\overrightarrow{\omega}$
satisfies the reverse H\"{o}lder's condition $RH_{\overrightarrow{p}},$ if
there exists a positive constant $C$ such that
\begin{equation}\label{RH cons}\prod\limits_{i=1}^{\infty}\big(\frac{1}{|B|}\int_B\sigma_idx\big)^{\frac{p}{p_i}}
\leq C\frac{1}{|B|}\int_B\prod\limits_{i=1}^{\infty}\sigma_i^{\frac{p}{p_i}}dx,~\forall B\in \mathcal{D}.\end{equation}
Moreover, we denote the smallest constants $C$ in \eqref{RH cons} by $[\overrightarrow{\omega}]_{RH_{\overrightarrow{p}}}$
\end{definition}

The following Lemma \ref{Carleson_lemma_multi} is a  formulation of the Carleson embedding theorem involving the infinite product. The linear one can be found in \cite{Hytonen}.

\begin{lemma}\label{Carleson_lemma_multi} Let $1<p_i<\infty, i\in N$ and $\frac{1}{p}=\sum\limits^{\infty}_{i=1}\frac{1}{p_i }.$
Let $\omega_i$ be weights and let $\sigma_i=\omega_i^{-\frac{1}{p_i-1}},~i\in N.$ Suppose that the nonnegative numbers $\{a_B\}_{B\in \mathcal{D}}$ satisfy

\begin{equation}\label{Carleson_assumption}
 \sum_{B\subset G} a_B \leq A \int_G \prod_{i=1}^{\infty} \sigma_i^{\frac{p}{p_i}}dx, \, \forall G \in \mathcal{D}.
\end{equation}
Then for all $\overrightarrow{f}\in \prod\limits_{i=1}^{\infty}{L^{p_i}(\sigma_i)}$, we have

\begin{equation}\label{Carleson}\begin{split}
\left(\sum_{B\in \mathcal{D}} a_B \Big(\prod_{i=1}^{\infty} \frac{1}{\sigma_i(B)}\int_{B}f_i(y_i)\sigma_i(y_i)d y_i\Big)^p\right)^{1/p} \leq& A^{\frac{1}{p}} ||\mathfrak{M}_d^{\overrightarrow{\sigma}}(\overrightarrow f)||_{L^p(\nu_{\overrightarrow \sigma})} \\
\leq& A^{\frac{1}{p}} \prod_{i=1}^{\infty} p'_i ||f_i||_{L^{p_i}(\sigma_i)},
\end{split}\end{equation}
where $\nu_{\overrightarrow \sigma}=\prod_{i=1}^{\infty} \sigma_i^{\frac{p}{p_i}}$ and $\mathfrak{M}_d^{\overrightarrow{\sigma}}(\overrightarrow{f})(x)=\sup\limits_{x\in B\in\mathcal{D}}\prod_{i=1}^{\infty} \frac{1}{\sigma_i(B)}\int_B f_i(y_i)\sigma_i(y_i)dy_i$.

\end{lemma}

\begin{proof}
  Let us see the sum
    $$\sum_{B\in\mathcal{D}} a_B \left(\prod_{i=1}^{\infty} \frac{1}{\sigma_i(B)}\int_B f_i(y_i)\sigma_i(y_i)dy_i \right)^p$$
  as an integral on a measure space $(\mathcal{D},2^{\mathcal{D}},\mu)$ built over the set of dyadic cubes $\mathcal{D}$, assigning to each $B\in\mathcal{D}$ the measure $a_B$. Thus

  \begin{equation*}\begin{split}
  &\sum_{B\in\mathcal{D}} a_B \left(\prod_{i=1}^{\infty} \frac{1}{\sigma_i(B)}\int_B f_i(y_i)\sigma_i(y_i)dy_i \right)^p = \\ &=
  \int_0^{\infty} p \lambda^{p-1} \mu\left\{B\in\mathcal{D}: \prod_{i=1}^{\infty} \frac{1}{\sigma_i(B)} \int_B f_i(y_i)\sigma_i(y_i)dy_i>\lambda\right\} \\ &=: \int_{0}^{\infty} p \lambda^{p-1}\mu(\mathcal{D}_{\lambda})d\lambda.
  \end{split}\end{equation*}

  Let us denote by $\mathcal{D}_{\lambda}(R)$ the dyadic cubes having side length $\leq R$ such that $$\prod_{i=1}^{\infty} \frac{1}{\sigma_i(B)}\int_B f_i(y_i)\sigma_i(y_i)dy_i >\lambda.$$
  Since all the cubes in $\mathcal{D}_{\lambda}(R)$ have side length less than or equal to $R,$ every cube will be contained in maximal one. Let $\mathcal{D}^*_{\lambda}(R)$ denote the subfamily formed by these maximal cubes.
Then the cubes $Q\in\mathcal{D}_{\lambda}^*(R)$ are disjoint and their union is contained in the set $\{\mathfrak{M}_d^{\overrightarrow{\sigma}}(\overrightarrow f)>\lambda\}$. Thus

  \begin{equation*}\begin{split}
    \mu(\mathcal{D}_{\lambda}(R)) &= \sum_{B\in\mathcal{D}_{\lambda}(R)} a_B \leq \sum_{Q\in\mathcal{D}_{\lambda}^{*}(R)}\sum_{B\subset Q} a_B \\
    &\leq A \sum_{Q\in\mathcal{D}^{*}_{\lambda}(R)} \int_Q \prod_{i=1}^{\infty} \sigma_i^{\frac{p}{p_i}} dx \\
    &\leq A \int_{\{\mathfrak{M}_{\overrightarrow{\sigma}}^d(\overrightarrow f)>\lambda\}} \prod_{i=1}^{\infty}\sigma_i^{\frac{p}{p_i}}dx.
  \end{split}\end{equation*}
It follows that $\mu(\mathcal{D}_{\lambda})\leq A \int_{\{\mathfrak{M}_{\overrightarrow{\sigma}}^d(\overrightarrow f)>\lambda\}} \prod_{i=1}^{\infty}\sigma_i^{\frac{p}{p_i}}dx.$
  Then we obtain

  \begin{equation*}\begin{split}
    \sum_{B\in\mathcal{D}} a_B \left(\prod_{i=1}^{\infty} \frac{1}{\sigma_i(B)}\int_B f_i(y_i)\sigma_i(y_i)dy_i \right)^p &\leq A\int_{0}^\infty p \lambda^{p-1} \int_{\{\mathfrak{M}^{\overrightarrow\sigma}_d(\overrightarrow f)>\lambda\}} \prod_{i=1}^{\infty} \sigma_i^{\frac{p}{p_i}} dx d\lambda \\
    &= A \int_{R_n} \mathfrak{M}^{\overrightarrow{\sigma}}_d(\overrightarrow f)^p \prod_{i=1}^{\infty} \sigma_i^{\frac{p}{p_i}}dx \\
    &\leq A \int_{R_n} \prod_{i=1}^{\infty} ((M^{\sigma_i}_d(f_i))^{p_i}\sigma_i)^{\frac{p}{p_i}}dx \\
    &\leq A \prod_{i=1}^{\infty} \left(\int_{R_n} (M^{\sigma_i}_d(f_i))^{p_i} \sigma_i dx \right)^{\frac{p}{p_i}} \\
    &\leq A \prod_{i=1}^{\infty} \left(p_i'\right)^p \left(\int_{R_n}f_i^{p_i}\sigma_i dx \right)^{\frac{p}{p_i}},
  \end{split}\end{equation*}
  where we have used that $\mathfrak{M}^{\overrightarrow{\sigma}}_d(\overrightarrow f)\leq\prod_{i=1}^{\infty} M^{\sigma_i}_d(f_i)$, the generalized H\"older's inequality and the boundedness properties of $M^{\sigma_i}_d$ in $L^{p_i}(\sigma_i).$
\end{proof}

\begin{theorem}\label{theorem_Sp} Let $1<p_i<\infty, i\in N$ and $\frac{1}{p}=\sum\limits^{\infty}_{i=1}\frac{1}{p_i }.$
If $(\omega_1, ~\omega_2,\cdot\cdot\cdot)\in RH_{\overrightarrow{p}},$
then the following statements
are equivalent:\begin{enumerate}
\item \label{thm Sp 1}There exists a positive constant $C$ such that
\be
\|\mathfrak{M}_d(\overrightarrow{f})\|_{L^p(v)}\leq
C\prod\limits^{\infty}_{i=1}\|f_i\|_{L^{p_i}(\omega_i)},
~\forall \overrightarrow{f}\in \prod\limits_{i=1}^{\infty}{L^{p_i}(\omega_i)}.
\ee

\item \label{thm Sp 3}There exists a positive constant $C$ such that
\be\big(\int_{B}\big(\mathfrak{M}_d(\overrightarrow{\sigma \chi_{B}})(x)\big)^pv(x)dx\big)^{\frac{1}{p}}
                       \leq C\prod\limits_{i=1}^{\infty}\big(\int_{B}\sigma_i(x)dx\big)^{\frac{1}{p_i}},~\forall B\in \mathcal{D}.
\ee
\end{enumerate}
Moreover, we denote the smallest constants $C$ in \eqref{thm Sp 1}
and \eqref{thm Sp 3} by $\|\mathfrak{M}_d\|$ and $[v,\overrightarrow{\omega}]_{S_{\overrightarrow{p}}},$ respectively.
Then it follows that
$$[v,\overrightarrow{\omega}]_{S_{\overrightarrow{p}}}
\leq\|\mathfrak{M}_d\|\leq [v,\overrightarrow{\omega}]_{S_{\overrightarrow{p}}}
[\overrightarrow{\omega}]_{RH_{\overrightarrow{p}}}^{\frac{1}{p}}.$$
\end{theorem}

\begin{proof} It is clear that $\eqref{thm Sp 1}\Rightarrow\eqref{thm Sp 3}$
without $(v,\overrightarrow{\omega})\in RH_{\overrightarrow{p}},$
so we omit it.

Next, assuming $\eqref{thm Sp 3},$ we shall prove $\eqref{thm Sp 1}.$
Let $\overrightarrow{f}\in\prod\limits_{i=1}^{\infty}L^{p_i}(\omega_i)$ and $\alpha>1.$ For every integer $k,$ we shall consider the set
$$S_k=\{x\in R_n: \alpha^k<\mathfrak{M}_d(\overrightarrow{f})(x)\leq \alpha^{k+1}\}.$$
From the definition of $\mathfrak{M}_d,$ $S_k\subseteq\bigcup\limits_{j}B_{k,j},$ where $B_{k,j}\in \mathcal{D}$ satisfies
$$\prod\limits_{i=1}^{\infty}\frac{1}{|B_{k,j}|}\int_{B_{k,j}}f_i(y_i)dy_i>\alpha^k.$$
Define $E_{k,1}=B_{k,1}\cap S_k$ and for $j>1:$
$$E_{k,j}=(B_{k,j}\backslash\bigcup\limits_{s<j}B_{k,s})\cap S_k.$$
The sets $S_k$ form a disjoint collection and each $S_k$ is disjoint union of the sets $E_{k,j}$ for varying $j.$
\be
   &~&\int_{R_n}\mathfrak{M}_d(\overrightarrow{f})^p v dx\\
   &\leq&\alpha^p\sum\limits_{k\in Z,j\in Z}v(E_{k,j})(\prod\limits_{i=1}^{\infty}\frac{1}{|B_{k,j}|}\int_{B_{k,j}}f_i(y_i)dy_i)^p\\
   &=&\alpha^p\sum\limits_{k\in Z,j\in Z}v(E_{k,j})(\prod\limits_{i=1}^{\infty}\frac{\sigma_i(B_{k,j})}{|B_{k,j}|})^p
      (\prod\limits_{i=1}^{\infty}\frac{1}{\sigma_i(B_{k,j})}\int_{B_{k,j}}f_i(y_i)\sigma_i(y_i)^{-1}\sigma_i(y_i)dy_i)^p\\
   &=&\alpha^p\sum\limits_{k\in Z,j\in Z}a_B
      (\prod\limits_{i=1}^{\infty}\frac{1}{\sigma_i(B)}\int_{B}f_i(y_i)\sigma_i(y_i)^{-1}\sigma_i(y_i)dy_i)^p,\ee
  where $a_B = v(E(B)) \left(\prod_{i=1}^{\infty} \frac{\sigma_i(B)}{|B|}\right)^p$, if $B=B_{k,j}$ for some $(k,j)$ where $E(B)$ denotes the corresponding set $E_{k,j}$ associated to $B_{k,j}$, and $a_B=0$ otherwise.  If we apply the Carleson embedding to these $a_B$, we will find the desired result provided that

  \begin{equation*}
    \sum_{B\subset G} a_B \leq A \int_G \prod_{i=1}^{\infty} \sigma_i^{\frac{p}{p_i}}dx, ~ G\in\mathcal{D}.
  \end{equation*}

  For $G\in\mathcal{D}$, we obtain

  \begin{equation*}\begin{split}
    \sum_{B\subset G} a_B &= \sum_{B_{k,j}\subset G} v(E_{k,j}) \left(\prod_{i=1}^{\infty} \frac{\sigma_i(B_{k,j})}{|B_{k,j}|}\right)^p \\
    &= \sum_{B_{k,j}\subset G} \int_{E_{k,j}} \left(\prod_{i=1}^{\infty} \frac{\sigma_i(B_{k,j})}{|B_{k,j}|}\right)^p v(x) dx \\
    &\leq \sum_{B_{k,j}\subset G} \int_{E_{k,j}}  (\mathfrak{M}_d(\overrightarrow{\sigma\chi_G}))^p v dx \\
    &\leq [v,\overrightarrow w]_{S_{\overrightarrow p}}^p \prod_{i=1}^{\infty} \sigma_i(G)^{\frac{p}{p_i}} \\
    &\leq [v,\overrightarrow w]_{S_{\overrightarrow p}}^p [\overrightarrow \omega]_{RH_{\overrightarrow p}} \int_G \prod_{i=1}^{\infty} \sigma_i^{\frac{p}{p_i}}dx,
  \end{split}\end{equation*}
  where in the next to last inequality we have used the $S_{\overrightarrow p}$ condition and in the last inequality we have used the $RH_{\overrightarrow p}$ condition. Thus, by Lemma \ref{Carleson_lemma_multi} we obtain that
  $$\|\mathfrak{M}_d(\overrightarrow{f})\|_{L^p(v)}\leq
\alpha[v,\overrightarrow w]_{S_{\overrightarrow p}} [\overrightarrow \omega]_{RH_{\overrightarrow p}}^{\frac{1}{p}}(\prod\limits_{i=1}^{\infty}p_i')\prod\limits^{\infty}_{i=1}\|f_i\|_{L^{p_i}(\omega_i)}.$$
Then we can take the limit $\alpha\rightarrow 1,$ which gives
$$(\int_{R_n}\mathfrak{M}_d(\overrightarrow{f})^p v dx)^{\frac{1}{p}}\leq [v,\overrightarrow{\omega}]_{S_{\overrightarrow{p}}}[\overrightarrow{\omega}]^{\frac{1}{p}}_{RH_{\overrightarrow{p}}}
                      (\prod\limits_{i=1}^{\infty}p_i')\prod\limits_{i=1}^{\infty}\|f_i\|_{L^{p_i}(\omega_i)}.$$
\end{proof}

%
%

\end{document}